\documentclass[12pt]{amsart}

\usepackage[margin=2.4cm]{geometry}
\usepackage{a4wide}

\usepackage{amsfonts,amsmath,amssymb,amsthm}

\newtheorem{thm}{Theorem}[section]
\newtheorem{lem}[thm]{Lemma}
\newtheorem{definition}{Definition}[section]

\newtheorem{example}[definition]{Example}

\begin{document}

\title[Lipschitzian solutions to iterative equations]{Lipschitzian solutions to inhomogeneous linear iterative equations}

\author[K. Baron]{Karol Baron}
\address{Instytut Matematyki\\
Uniwersytet \'{S}l\c aski\\
Bankowa 14, PL-40-007 Katowice\\
Poland}
\email{baron@us.edu.pl}

\author[J. Morawiec]{Janusz Morawiec}
\address{Instytut Matematyki\\
Uniwersytet \'{S}l\c aski\\
Bankowa 14, PL-40-007 Katowice\\
Poland}
\email{morawiec@math.us.edu.pl}


\begin{abstract}
We study the problems of the existence, uniqueness and continuous dependence of Lipschitzian solutions $\varphi$ of equations of the form
$$
\varphi(x)=\int_{\Omega}g(\omega)\varphi\big(f(x,\omega)\big)\mu(d\omega)+F(x),
$$
where $\mu$ is a measure on a  $\sigma$-algebra of subsets of $\Omega$.
\end{abstract}

\maketitle

%%%%%%%%%%%% 1 %%%%%%%%%%%%

\section{Introduction}

Fix a measure space $(\Omega,{\mathcal A},\mu)$ and a separable metric space $(X,\rho)$.

Motivated by appearance of the equation
$$
\varphi(x)=\int_{A_1}\varphi\big(f(x,\omega)\big)\mu(d\omega)+c-\int_{A_2}\varphi\big(f(x,\omega)\big)\mu(d\omega)
$$
with disjoint $A_1, A_2 \in \mathcal A$ in the theory of perpetuities and of refinement equations, see section 3.4 of the survey paper \cite{KM2013}, we consider problems of the existence, uniqueness and continuous dependence of Lipschitzian solutions $\varphi$ to the equation
\begin{equation}\label{1}
\varphi(x)=\int_{\Omega}g(\omega)\varphi\big(f(x,\omega)\big)\mu(d\omega)+F(x).
\end{equation}
Concerning the given functions $f, g$ and $F$ we assume the following hypotheses in which ${\mathcal B}$ stands for the $\sigma$-algebra of all Borel subsets of $X$ and $\mathbb K \in \{\mathbb R, \mathbb C\}$.

(H$_1$) Function $f$ maps $X\times\Omega$ into $X$ and for every $x\in X$ the function $f(x,\cdot)$ is $\mathcal A$-measurable, i.e., 
$$
\big\{\omega\in\Omega:f(x,\omega)\in B\big\}\in{\mathcal A}\hspace{3ex}\hbox{ for all }x\in X\hbox{ and }B\in{\mathcal B}.
$$

(H$_2$) Function $g\colon\Omega\to \mathbb K$ is integrable,
$$
\int_{\Omega}|g(\omega)|\rho\big(f(x,\omega),x\big)\mu(d\omega)<\infty
\hspace{3ex}\hbox{ for every }x\in X,
$$
and
\begin{equation}\label{b1}
\int_{\Omega}|g(\omega)|\rho\big(f(x,\omega),f(z,\omega)\big)\mu(d\omega)\leq\lambda\rho(x,z)
\hspace{3ex}\hbox{ for all }x,z\in X
\end{equation}
with a $\lambda\in[0,1)$.

(H$_3$) Function $F$ maps $X$ into a separable Banach space $Y$ over $\mathbb K$ and
\begin{equation}\label{2}
\|F(x)-F(z)\|\leq L\rho(x,z)\hspace{3ex}\hbox{ for all }x,z\in X
\end{equation}
with an $L\in[0,+\infty)$.

As emphasized in \cite[section 0.3]{KCG1990} iteration is the fundamental technique for solving functional equations in a single variable, and iterates usually appear in the formulae for solutions. However, as it seams, Lipschitzian solutions are examined rather by the fixed-point method (cf. \cite[section 7.2D]{KCG1990}). We iterate the operator which transforms a Lipschitzian  $F\colon X\to Y$ into 
$\int_{\Omega}g(\omega)F\big(f(x,\omega)\big)\mu(d\omega)$; cf. formulas (\ref{3}) and (\ref{6}) below.
The spacial case where $g(\omega)=-1$ for every $\omega \in \Omega$ and $\mu(\Omega)=1$ was examined in \cite{BKM2015} on a base of iteration of random-valued functions.

Integrating vector functions we use the Bochner integral.

%%%%%%%%%%%% 2 %%%%%%%%%%%%

\section{Existence and uniqueness}
Putting
\begin{equation}\label{5}
\gamma=\int_{\Omega}g(\omega)\mu(d\omega),
\end{equation}
we start with two simple lemmas.

\begin{lem}\label{lem21}
Assume {\rm (H$_1$)} and let $g\colon\Omega\to\mathbb K$  be integrable with $\gamma\neq 1$. If $(\ref{b1})$ holds with a $\lambda\in[0,1)$, then for any $F$ mapping $X$ into a normed space $Y$ over $\mathbb K$ equation $(\ref{1})$ has at most one Lipschitzian solution $\varphi\colon X\to Y$.
\end{lem}

\begin{proof}
Fix a function $F$ mapping $X$ into a normed space $Y$ over $\mathbb K$, let $\varphi_1,\varphi_2\colon X\to Y$ be Lipschitzian solutions of (\ref{1}), and put $\varphi=\varphi_1-\varphi_2$. Then $\varphi$ is a Lipschitzian solution of (\ref{1}) with $F=0$, and denoting by $L_\varphi$ the smallest Lipschitz constant for $\varphi$, by (\ref{b1})  for all $x,z\in X$ we have
$$
\|\varphi(x)-\varphi(z)\|\leq\int_{\Omega}|g(\omega)|\big\|\varphi\big(f(x,\omega)\big)-\varphi\big(f(z,\omega)\big)\big\|\mu(d\omega)\leq L_\varphi\lambda \rho(x,z),
$$
whence $L_\varphi=0$ and $\varphi$ is a constant function. Since $\gamma$ defined by (\ref{5}) is different from 1, the only constant solution of (\ref{1}) with $F=0$ is the zero function. 
\end{proof}

\begin{lem}\label{lem22}
Under the assumptions {\rm (H$_1$)--(H$_3$)} for every $x\in X$ the function 
$$
\omega\mapsto g(\omega)F\big(f(x,\omega)\big),\hspace{3ex}\omega\in\Omega,
$$ 
is Bochner integrable and
\begin{equation}\label{b2}
\Big\|\!\int_{\Omega}g(\omega)F\big(f(x,\omega)\big)\mu(d\omega)-\int_{\Omega}g(\omega)F\big(f(z,\omega)\big)\mu(d\omega)\Big\|\leq L\lambda\rho(x,z)
\end{equation}
for all $x,z\in X$.
\end{lem}

\begin{proof}
The function considered is $\mathcal A$-measurable, for every $\omega\in\Omega$ we have
$$
\big\|g(\omega)F\big(f(x,\omega)\big)\big\|\leq L|g(\omega)|\rho\big(f(x,\omega),x\big)+L|g(\omega)|\|F(x)\|,
$$
and (\ref{b2}) holds for all $x,z\in X$.
\end{proof}

Assuming (H$_1$)--(H$_3$) and applying Lemma \ref{lem22} we define
\begin{equation}\label{3}
F_0(x)=F(x), \hspace{3ex} F_n(x)=\int_{\Omega}g(\omega)F_{n-1}\big(f(x,\omega)\big)\mu(d\omega) 
\end{equation} 
for all $x \in X$ and $n \in \mathbb N$, and we see that
\begin{equation}\label{4}
\|F_n(x)-F_n(z)\|\leq L\lambda^n\rho(x,z)\hspace{3ex}\hbox{ for all }x,z\in X \hbox{ and } n \in \mathbb N.
\end{equation}

Our main result reads.

\begin{thm}\label{thm23}
Assume {\rm (H$_1$)--(H$_3$)}. If $\gamma \not= 1$ then equation $(\ref{1})$ has exactly one Lipschitzian solution $\varphi\colon X\to Y$; it is given by the formula 
\begin{equation}\label{6}
\varphi(x)= \frac{1}{1-\gamma}\left(\sum_{n=1}^{\infty}\big(F_n(x)-\gamma F_{n-1}(x)\big)+F(x)\right)\hspace{3ex}\hbox{ for every }x \in X,
\end{equation}
\begin{equation}\label{7}
\|\varphi(x)-\varphi(z)\|\leq\frac{L(1+|\gamma|)}{|1-\gamma|(1-\lambda)}\rho(x,z)\hspace{3ex}\hbox{ for all } x,z\in X,
\end{equation}
and
\begin{equation}\label{8}
\|\varphi(x)\|\leq \frac{1}{|1-\gamma|}\left(\frac{L}{1-\lambda}\int_\Omega|g(\omega)|\rho\big(f(x,\omega),x\big)\mu(d\omega)+\|F(x)\|\right)
\end{equation}
for every $x\in X$.
\end{thm}

\begin{proof}
For the proof of the existence observe first that by (\ref{5}), (\ref{3}) and (\ref{4}) for all $x \in X$ and $n \in \mathbb N$ we have
\begin{eqnarray}\label{9}
\nonumber
\|F_n(x)-\gamma F_{n-1}(x)\|\!
&=&\!\Big\|\!\int_{\Omega}\!g(\omega)F_{n-1}\big(f(x,\omega)\big)\mu(d\omega)-\! \int_{\Omega}\!g(\omega)F_{n-1}(x)\mu(d\omega)\Big\|\\
&\leq& L\lambda^{n-1}\int_{\Omega}|g(\omega)|\rho\big(f(x,\omega),x\big)\mu(d\omega).
\end{eqnarray}
Consequently (\ref{6}) defines a function $\varphi\colon X\to Y$. Routine calculations, (\ref{6}), (\ref{4}), (\ref{b1}) and (\ref{9}) show that this function satisfies (\ref{7}) and (\ref{8}).

It remains to prove that $\varphi$ solves (\ref{1}). To this end define $M\colon X\to[0,\infty)$ by
\begin{equation}\label{10}
M(x)=L\int_\Omega|g(\omega)|\rho\big(f(x,\omega),x\big)\mu(d\omega)
\end{equation}
and fix $x_0 \in X$. An obvious application of (\ref{10}), (H$_2$), (\ref{8}) and (\ref{2}) gives 
\begin{equation}\label{11}
M(x)\leq c_1\rho (x,x_0)+c_2, \hspace{3ex}\|\varphi(x)\|\leq c_1\rho(x,x_0)+c_2
\hspace{3ex}\hbox{for every }x\in X 
\end{equation}
with some constants $c_1,c_2\in[0,\infty)$.

Fix $x \in X$. According to Lemma \ref{lem22} the function
$$
\omega \longmapsto g(\omega)\varphi\big(f(x,\omega)\big),\hspace{3ex}\omega \in \Omega,
$$
is Bochner integrable. Moreover, by (\ref{9})--(\ref{11}),
\begin{eqnarray*} 
\hspace*{15ex}&&\hspace*{-25ex}\Big\|g(\omega)\Big(F_n\big(f(x,\omega)\big)-\gamma F_{n-1}\big(f(x,\omega)\big)\Big)\Big\|\leq\lambda^{n-1}|g(\omega)|M\big(f(x,\omega)\big)\\
&\leq& \lambda^{n-1}|g(\omega)|\big(c_1\rho(f(x,\omega),x_0)+c_2\big)\\
&\leq& \lambda^{n-1}|g(\omega)|\big(c_1\rho(f(x,\omega),x)+c_1\rho(x,x_0)+c_2\big)
\end{eqnarray*}
for all $n \in \mathbb N$ and $\omega \in \Omega$. Hence, making use of (H$_2$), the dominated convergence theorem and (\ref{3}) we see that
\begin{eqnarray}\label{13} 
\nonumber
\hspace*{10ex}&&\hspace*{-20ex}
\int_\Omega\sum_{n=1}^\infty g(\omega)\Big(F_n\big(f(x,\omega)\big)-\gamma F_{n-1}\big(f(x,\omega)\big)\Big)\mu(d\omega)\\
&=&\sum_{n=1}^\infty\int_\Omega g(\omega)\Big(F_n\big(f(x,\omega)\big)-\gamma F_{n-1}\big(f(x,\omega)\big)\Big)\mu(d\omega)\\
\nonumber
&=&\sum_{n=1}^\infty\big(F_{n+1}(x)-\gamma F_n(x)\big).
\end{eqnarray}
Applying now (\ref{6}), (\ref{13}) and (\ref{3}) we get 
\begin{eqnarray*} 
\hspace*{10ex}&&\hspace*{-20ex}\int_{\Omega}g(\omega)\varphi\big(f(x,\omega)\big)\mu(d\omega)\\
&=&
\frac{1}{1-\gamma}\int_{\Omega}\Bigg[\sum_{n=1}^{\infty}g(\omega)\Big(F_n\big(f(x,\omega)\big)
-\gamma F_{n-1}\big(f(x,\omega)\big)\Big)\\
&&+g(\omega)F\big(f(x,\omega)\big)\Bigg]\mu(d\omega) \\
&=& \frac{1}{1-\gamma}\left[\sum_{n=1}^\infty\Big(F_{n+1}(x)-\gamma F_n(x)\Big)+F_1(x)\right]\\
&=& \frac{1}{1-\gamma}\left[\sum_{n=1}^\infty\Big(F_n(x)-\gamma F_{n-1}(x)\Big)+\gamma F(x)\right]\\
&=& \varphi(x)-F(x).
\end{eqnarray*}
The proof is complete.
\end{proof}

%%%%%%%%%%%% 3 %%%%%%%%%%%%

\section{Examples}

\begin{example}\label{ex31}
{\rm Given $\lambda\in(0,1)$ and an integrable $\xi\colon\Omega\to\mathbb R$ consider the equation
$$
\varphi(x)=\lambda^2\int_{\Omega}\varphi\left(\frac{1}{\lambda}x+\xi(\omega)\right)\mu(d\omega)
$$
with $\mu(\Omega)=1$. According to Lemma \ref{lem21} the zero function is the only its Lipschitzian solution $\varphi\colon\mathbb R\to\mathbb R$. Note however that if
$$
\int_{\Omega}\xi(\omega)\mu(d\omega)=0\hspace{3ex}\hbox{ and }\hspace{3ex}
\int_{\Omega}\xi(\omega)^2\mu(d\omega)<\infty,
$$
then this equation solves also the function $\varphi\colon\mathbb R\to\mathbb R$ given by
$$
\varphi(x)=x^2+\frac{\lambda^2}{1-\lambda^2}\int_{\Omega}\xi(\omega)^2\mu(d\omega).
$$}
\end{example}

\begin{example}\label{ex32}
{\rm Given $\lambda\in(0,1)$ consider the equation
$$
\varphi(x)=2\varphi\big(\lambda\sqrt{x}+1-\lambda\big)+\log\frac{x}{(\lambda\sqrt{x}+1-\lambda)^2}.
$$
According to Lemma \ref{lem21} (in this case $f(x,\omega)=\lambda\sqrt{x}+1-\lambda$, $g(\omega)=2$ and $F(x)=\log\frac{x}{(\lambda\sqrt{x}+1-\lambda)^2}$ for all $x\in[1,\infty)$ and $\omega\in\Omega$, $\mu(\Omega)=1$) the logarithmic function restricted to $[1,\infty)$ is the only Lipschitzian solution $\varphi\colon[1,\infty)\to\mathbb R$ to this equation, and it is unbounded in spite of the fact that $F$ is bounded.}
\end{example}

\begin{example}\label{ex33}
{\rm To see that assumptions (H$_1$)--(H$_3$) do not guarantee the existence of a {\it continuous} solution   $\varphi\colon X\to Y$ to equation (\ref{1}), given $\alpha\in(-1,1)$, a bounded and $\mathcal A$-measurable $\xi\colon\Omega\to\mathbb R$, and a Lipschitzian $F\colon\mathbb R\to[0,\infty)$ such that $F^{-1}(\{0\})$ is a singleton, consider the equation 
\begin{equation}\label{b3}
\varphi(x)=\int_{\Omega}\varphi\big(\alpha x+\xi(\omega)\big)\mu(d\omega)+F(x)
\end{equation}
with $\mu(\Omega)=1$. Assume a continuous $\varphi\colon\mathbb R\to\mathbb R$ solves it. We shall see that then $\xi$ is a.e. constant. To this end fix an $M\in(0,\infty)$ such that $|\xi(\omega)|\leq M$ for every $\omega\in\Omega$, and a real number $a\geq\frac{M}{1-|\alpha|}$ such that $F^{-1}(\{0\})\subset[-a,a]$. Then
$$
|\alpha x+\xi(\omega)|\leq a\hspace{3ex}\hbox{ for all }x\in[-a,a]\hbox{ and }\omega\in\Omega
$$
and so $\varphi |_{[-a,a]}$ is a continuous, hence also bounded, solution of (\ref{b3}). According to \cite[Corollary 4.1(ii) and Example 4.1]{B2009} it is possible only if $\xi$ is a.e. constant.}
\end{example}

%%%%%%%%%%%% 4 %%%%%%%%%%%%

\section{Continuous dependence}

Given a normed space $(Y,\|\cdot\|)$ consider now the linear space $Lip(X,Y)$ of all Lipschitzian functions mapping $X$ into $Y$, and its linear subspace $BL(X,Y)$ of all Lipschitzian and bounded functions mapping $X$ into $Y$.  Fix $x_0 \in X$ and define $\|\cdot\|_{Lip}\colon Lip(X,Y)\to [0,\infty)$ by
$$
\|u\|_{Lip}=\|u(x_0)\|+\|u\|_L,
$$
where $\|u\|_L$  stands for the smallest Lipschitz constant for $u$. Clearly $\|\cdot\|_{Lip}$ is a norm in $Lip(X,Y)$. It depends on the fixed point $x_0$, but for different points such norms are equivalent. It is well known that if $(Y,\|\cdot\|)$ is Banach, then so is $(Lip(X,Y),\|\cdot\|_{Lip})$. In the linear space $BL(X,Y)$ we consider the norm $\|\cdot\|_{BL}$ given by
$$
\|u\|_{BL}=\sup\big\{\|u(x)\|:x\in X\big\}+\|u\|_L.
$$
It is also well known that if $(Y,\|\cdot\|)$ is Banach, then so is $(BL(X,Y),\|\cdot\|_{BL})$.

Assume (H$_1$) and (H$_2$), $\gamma \not= 1$, and let $Y$ be a separable Banach space over $\mathbb K$. 

According to Theorem $\ref{thm23}$ for every $F\in Lip(X,Y)$ the formula
\begin{equation}\label{14}
\varphi^F(x)= \frac{1}{1-\gamma}\left(\sum_{n=1}^{\infty}\Big(F_n(x)-\gamma F_{n-1}(x)\Big)+F(x)\right)
\end{equation}
for every $x \in X$, defines the only Lipschitzian solution $\varphi^F$ of equation (\ref{1}),
\begin{equation}\label{15}
\|\varphi^F\|_L\leq\frac{1+|\gamma|}{|1-\gamma|(1-\lambda)}\|F\|_L
\end{equation}
and
\begin{equation}\label{16}
\|\varphi^F(x)\|\leq \frac{1}{|1-\gamma|}\left(\frac{\|F\|_L}{1-\lambda}\int_\Omega|g(\omega)|\rho\big(f(x,\omega),x\big)\mu(d\omega)+\|F(x)\|\right)
\end{equation}
for every $x \in X$. Putting
\begin{equation}\label{17}
c_0=\frac{1}{1-\lambda}\left(\int_\Omega|g(\omega)|\rho\big(f(x_0,\omega),x_0\big)\mu(d\omega)+1+|\gamma|\right),\hspace{3ex} c=\max\{1,c_0\},
\end{equation}
and applying (\ref{15}) and (\ref{16}) we see that if $F \in Lip(X,Y)$, then
\begin{eqnarray*}
\|\varphi^F\|_{Lip}&=&\|\varphi^F(x_0)\|+\|\varphi^F\|_L\leq\frac{1}{|1-\gamma|}\big(c_0\|F\|_L+\|F(x_0)\|\big)\\
&\leq&\frac{c}{|1-\gamma|}\|F\|_{Lip}.
\end{eqnarray*}
Moreover, if $d_0$ defined by
\begin{equation}\label{b4}
d_0=\sup\left\{\int_{\Omega}|g(\omega)|\rho\big(f(x,\omega),x\big)\mu(d\omega):x\in X\right\}
\end{equation}
is finite, then putting
\begin{equation}\label{b5}
d=\max\left\{1,\frac{d_0+1+|\gamma|}{1-\lambda}\right\}
\end{equation}
and applying (\ref{16}) and (\ref{15})  again we see also that if $F\in BL(X,Y)$, then $\varphi^F\in BL(X,Y)$ as well and
$$
\|\varphi^F\|_{BL}\leq\frac{1}{|1-\gamma|}\!\left(\frac{d_0+1+|\gamma|}{1-\lambda}\|F\|_{L}+\sup\big\{\|F(x)\|\!:\!x\in X\big\}\!\right)\leq\frac{d}{|1-\gamma|}\|F\|_{BL}.
$$

\begin{thm}\label{thm41}
Assume {\rm (H$_1$), (H$_2$)} and let $\gamma$ defined by $(\ref{5})$ be different from {\rm 1}. If $Y$ is a separable Banach space over $\mathbb K$ then:

{\rm (i)} for any $F\in Lip(X,Y)$ the function $\varphi^F\colon X \to Y$ defined by $(\ref{14})$ and $(\ref{3})$ is the only Lipschitzian solution of $(\ref{1})$, the operator
\begin{equation}\label{18}
F\mapsto \varphi^F,\hspace{3ex}F\in Lip(X,Y),
\end{equation}
is a linear homeomorphism of  $(Lip(X,Y),\|\cdot\|_{Lip})$ onto itself  and
$$
\|\varphi^F\|_{Lip}\leq\frac{c}{|1-\gamma|}\|F\|_{Lip}\hspace{3ex}\hbox{ for every } F\in Lip(X,Y)
$$
with $c$ given by $(\ref{17})$;

{\rm (ii)} if additionally $d_0$ defined by $(\ref{b4})$ is finite, then the restriction of the operator $(\ref{18})$ to $BL(X,Y)$ is a linear homeomorphism of $(BL(X,Y),\|\cdot\|_{BL})$ onto itself and
$$
\|\varphi^F\|_{BL}\leq\frac{d}{|1-\gamma|}\|F\|_{BL}\hspace{3ex}\hbox{ for every }F\in BL(X,Y)
$$
with $d$ given by $(\ref{b5})$.
\end{thm}

\begin{proof}
By the above considerations and the Banach inverse mapping theorem it remains to show that operator (\ref{18}) is  one-to-one, maps $Lip(X,Y)$ onto $Lip(X,Y)$ and $BL(X,Y)$ onto $BL(X,Y)$.

The first property follows from the fact that for any $F\in Lip(X,Y)$ the function $\varphi^F$ is a solution of (\ref{1}): if $\varphi^F=0$, then $F=0$. To get the next two observe that if $\psi\in Lip(X,Y)$, then by Lemma \ref{lem22} the function $F\colon X\to Y$ given by
$$
F(x)=\psi(x)-\int_\Omega g(\omega)\psi (f(x,\omega))\mu(d\omega)
$$
belongs to $Lip(X,Y)$, if $\psi$ is also bounded, then so is $F$, and, since both $\psi$ and $\varphi^F$ solve (\ref{1}), $\psi=\varphi^F$ by Lemma \ref{lem21}.
\end{proof}

\section*{Acknowledgement}
This research was supported by the University of Silesia Mathematics Department (Iterative Functional Equations and Real Analysis program).

\end{document}